\documentclass[reqno,12pt,letterpaper]{amsart}
\usepackage{amsmath,amssymb,amsthm,graphicx,mathrsfs,url}
\usepackage[usenames,dvipsnames]{color}
\usepackage[colorlinks=true,linkcolor=Red,citecolor=Green]{hyperref}
\usepackage{amsxtra}

\def\?[#1]{\textbf{[#1]}\marginpar{\Large{\textbf{??}}}}

\def\smallsection#1{\smallskip\noindent\textbf{#1}.}
\let\epsilon=\varepsilon 

\newcommand{\RR}{{\mathbb R}}
\newcommand{\NN}{{\mathbb N}}
\newcommand{\CC}{{\mathbb C}}
\newcommand{\SP}{{\mathbb S}}
\newcommand{\CI}{C^\infty}
\newcommand{\CIc}{C^\infty_{\rm{c}}}

\newtheorem{thm}{Theorem}
\newtheorem{prop}{Proposition}[section]

\newtheorem{lemm}[prop]{Lemma}

\numberwithin{equation}{section}

\DeclareMathOperator{\Ai}{Ai}
\DeclareMathOperator{\Res}{Res}

\DeclareMathOperator{\comp}{comp}

\let\Im=\Imag
\DeclareMathOperator{\loc}{loc}

\let\Re=\Real

\DeclareMathOperator{\supp}{supp}

\def\indic{\operatorname{1\hskip-3.15pt\relax l}}

\title{Resonances and lower resolvent bounds}

\author{Kiril Datchev}
\email{datchev@math.mit.edu}
\author{Semyon Dyatlov}
\email{dyatlov@math.mit.edu}
\address{Department of Mathematics, Massachusetts Institute of Technology,
Cambridge, MA 02139}
\author{Maciej Zworski}
\email{zworski@math.berkeley.edu}
\address{Department of Mathematics, University of California,
Berkeley, CA 94720, USA}

\begin{document}

\begin{abstract}We show how the presence of resonances close to the real
axis implies exponential lower bounds on the norm of the cut-off 
resolvent on the real axis.
\end{abstract}

\maketitle

\addtocounter{section}{1}
\addcontentsline{toc}{section}{1. Introduction}

In this note we establish exponential lower bounds on the scattering resolvent on
the real line. We show that these lower bounds can be understood in
terms of resonances close to the real axis.

To fix the concepts, consider a \emph{semiclassical Schr\"odinger operator}
on $\mathbb R^n$:
\begin{equation}
  \label{e:p-s}
P(h)=-h^2\Delta+V(x),\quad x\in\mathbb R^n,\quad
V\in \CIc(\mathbb R^n;\mathbb R), 
\end{equation}
\begin{equation}
  \label{e:as-intro}
\begin{gathered}
\supp V\subset B(0,R_0);\quad
V(0)=V_0>0,\ V'(0)=0,\
V''(0)>0;\\
x\cdot V'(x)\leq 0\quad\text{on }\{V\leq V_0\},\quad
x\cdot V'(x)<0\quad\text{on }\{V=V_0\}\setminus \{0\}.
\end{gathered}
\end{equation}
Take $R>R_0$ and define the cutoffs
\begin{gather}
\label{eq:cutoffs}
\begin{gathered}
\chi=\indic_{B(0,R_0)},\quad
\psi=\indic_{B(0,R+1)\setminus B(0,R-1)}.
\end{gathered}
\end{gather} 
Theorem \ref{t:2} in \S\ref{s:ame} 
shows that for any $ R > R_0 $ there exists a constant $c>0$ independent of $h$
and $ E_0 (h) = V_0 + \mathcal O ( h ) $ such that
\begin{equation}
\label{eq:res11}
\|  \chi ( P(h) - E_0 ( h) \pm i 0 )^{-1} \chi \|_{L^2 \to L^2 }
\geq \exp (  c /h ) ,
\end{equation}
\begin{equation}
\label{eq:res21}
\|  \psi ( P(h)  -  E_0 ( h )  \pm i 0 )^{-1} \chi \|_{L^2 \to L^2 }
\geq \exp( c  / h). 
\end{equation}
A very general exponential {\em upper} bound 
corresponding to \eqref{eq:res11} was first proved by Burq \cite{B},
with generalizations by Vodev \cite{v}, and more recently by Datchev
\cite{D}. The lower bound is immediate from much easier arguments
involving quasimodes. The ``non-trapping'' upper bound (for $R$ large enough) 
\begin{equation}
\label{eq:res31}
\|  \psi ( P(h)  - E_0 ( h ) \pm i 0 )^{-1} \psi\|_{L^2 \to L^2 }
\leq \frac{ C_0 }{  h}, 
\end{equation}
was again given by Burq \cite{B} (with a $ \log 1/h
$ loss) and Vodev \cite{v} -- see \cite{D} for a neat new proof. 

It is \eqref{eq:res21} which seems to be the novel aspect. It shows
that having a one sided cutoff to the exterior of the interaction 
region {\em cannot} prevent exponential blow up of the resolvent.

The method also applies to the case of Riemannan manifolds, $ (M , g ) $, considered recently by
Rodnianski--Tao \cite{rt} -- see Fig.~\ref{f:3}. In that case  
the support of~$ V $  is replaced in~\eqref{eq:cutoffs} by the set where the 
metric is different from the Euclidean metric, and we 
obtain a sequence of $ \lambda_k \to \infty $ such that
\begin{equation}
\label{eq:rota1}
\| \psi (  \Delta_g - \lambda_k \pm i 0 ) ^{-1} \chi \|_{L^2 ( M ) \to
  L^2 ( M ) } \geq e^{ c\sqrt {\lambda_k } } . 
\end{equation}
See Theorem~\ref{t:3} in~\S\ref{s:ame} for details.

\begin{figure}
\begin{center}
\includegraphics[width=2.2in]{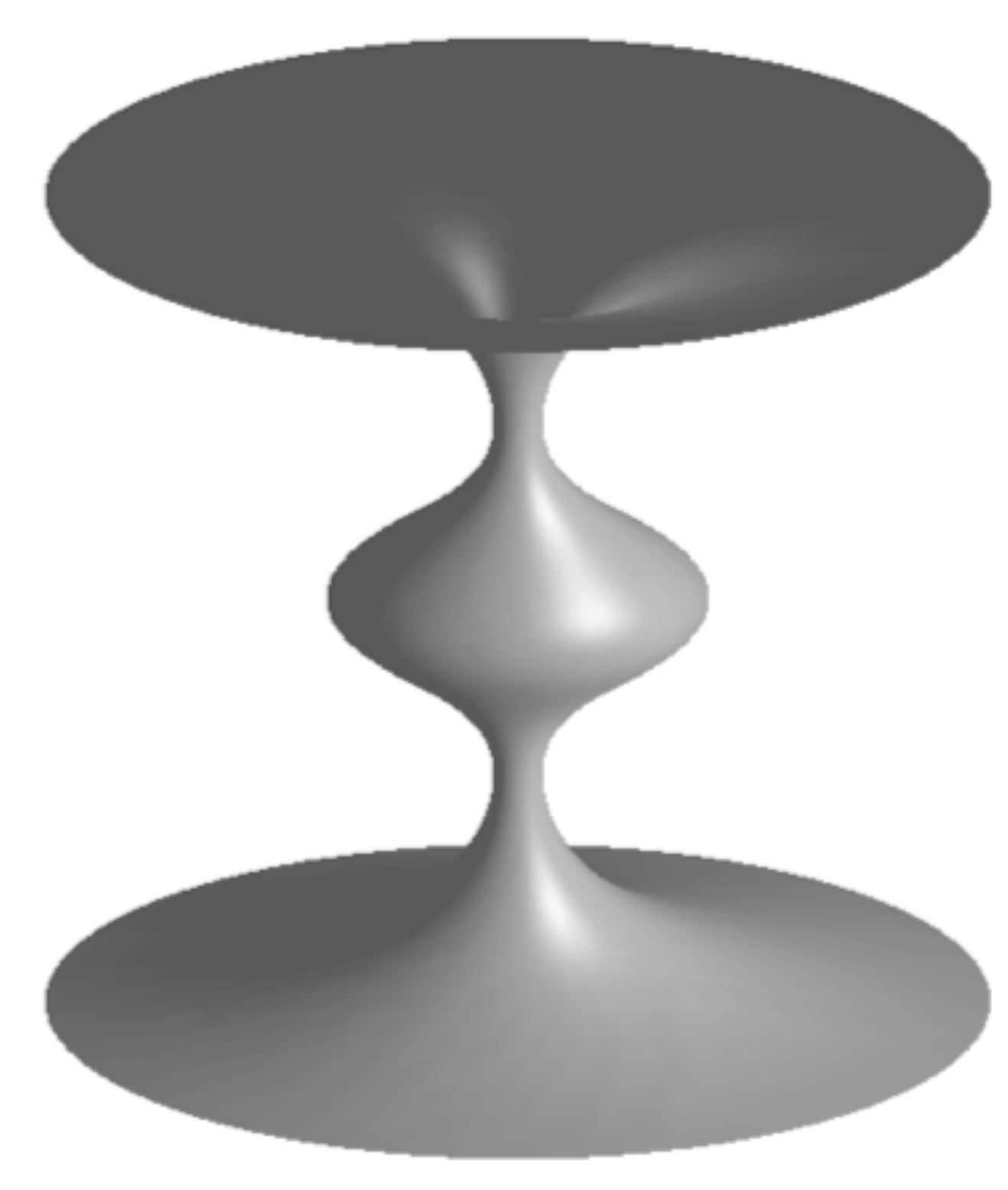} \hspace{0.5in} \includegraphics[width=2.5in]{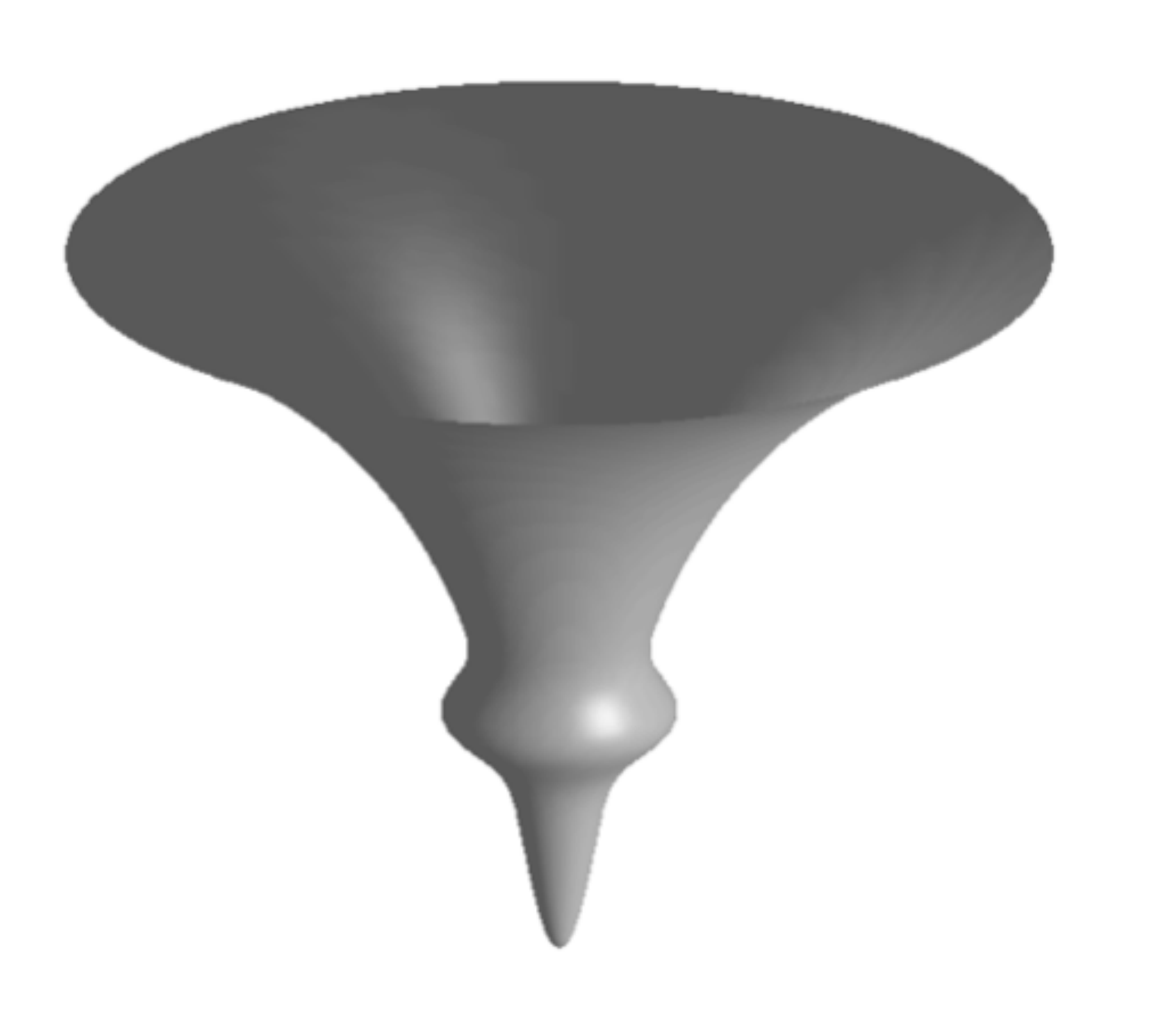}
\end{center}
\caption{Examples of manifolds for which the estimate
  \eqref{eq:rota1} holds: on the left a surface of revolution
with two Euclidean ends to which Theorem \ref{t:3} applies directly;
on the right a surface with one end to which a modification 
of the same method applies (see \cite{omg}). The same
examples work in any dimension.
\label{f:3}}
\end{figure}

The reason behind these estimates is the presence of resonances close
to real axis.
The resolvent
$$
R_h(z)=(P(h)-z)^{-1}:L^2(\mathbb R^n)\to H^2(\mathbb R^n),\quad
z\not\in [0,+\infty),
$$
has a meromoprhic continuation to the Riemann surface of $\sqrt{z}$ for $n$ odd,
and to the Riemann surface of $\log z$ for $n$ even, as a family of operators
$L^2_{\comp}(\mathbb R^n)\to H^2_{\loc}(\mathbb R^n)$, see for
example~\cite{res} and references given there.
\emph{Resonances}, defined as the
poles of this continuation,
replace discrete spectral data for problems on non-compact
domains. For the situations considered here, in particular
for the case~\eqref{e:p-s}, Rellich's theorem (see~\cite{res}) shows that there are no
resonances on the positive real axis, that is, the
operators $(P(h)-E\pm i0)^{-1}$ are well-defined for $E>0$.

Theorem \ref{t:1} in~\S\ref{s:general} gives general lower bounds based on existence
of resonances with certain properties. It is then applied in Theorems
\ref{t:2} and \ref{t:3} in~\S\ref{s:ame} to obtain examples, in particular of
Riemannian manifolds with Euclidean ends.

The simple proofs here are based on previous work on scattering resonances, in
particular those by 
Bony--Michel \cite{bm}, G\'erard--Martinez \cite{gm},
Helffer--Sj\"ostrand \cite{hsj}, Tang-Zworski \cite{tz} and  Nakamura--Stefanov--Zworski \cite{nsz}.
To make the basic idea accessible,
we present in \S\ref{ele} an elementary and self-contained one dimensional
example which captures the basic reason for \eqref{eq:res21}; the argument
of~\S\ref{ele} does not directly use resonances though it could be
used to show their existence.



\smallsection{Acknowledgements} 
We would like to thank Andr\'e Martinez and Andr\'as Vasy for
helpful discussions of resolvent estimates. 
We are also grateful for the support by 
a National Science Foundation postdoctoral fellowship (KD), 
a Clay Research Fellowship (SD) and by 
a National Science Foundation grant DMS-1201417 (MZ).

\section{An explicit example}
\label{ele}

The following one dimensional example shows the reasons for
\eqref{eq:res21} in an explicit setting:
\begin{thm}
  \label{t:1d}
Let $V\in \CIc(\mathbb R)$ be a nonnegative potential satisfying the following conditions (see Figure~\ref{f:V}):
\begin{gather}
\label{eq:potV}
\begin{gathered} 
V ( x ) = V ( -x ) , \ \ \
\text{$ V(x)=x^2+1$ \ for $x\in [-1,1]$;} \\
\text{$V(x)=4-x$ for $x\in [2, 3.5]$; \ \ 
$V(x)<1$ for $x> 3$;\ \ 
$\supp V\subset [-5,5]$.}
\end{gathered}
\end{gather}
Put $R_0=4$, fix $R>5$, and define $\chi,\psi$ by~\eqref{eq:cutoffs}.
Then there exists $c>0$
and families $E_0(h)=1+\mathcal O(h)$,
$u(h),f(h)\in \CIc(\mathbb R)$ such that
\begin{equation}
  \label{e:1d}
\begin{gathered}
(P(h)-E_0(h))u=f,\quad
f=\psi f;\\
\|\chi u\|_{L^2(\RR)}=1,\quad
\|f\|_{L^2(\RR)}\leq e^{-c/h}.
\end{gathered}
\end{equation}
\end{thm}
Note that~\eqref{e:1d} implies (henceforth suppressing the dependence on $h$)
\begin{gather}
  \label{e:lower-bound-1}
\|\chi R_h(E_0\pm i0)\psi\|_{L^2\to L^2}\geq e^{c/h},\\
  \label{e:lower-bound-2}
\|\psi R_h(E_0\pm i0)\chi\|_{L^2\to L^2}\geq e^{c/h}.  
\end{gather}
Indeed,
since $u\in \CIc$ and $f=\psi f$, we have $\chi u=\chi R_h(E_0\pm i0)\psi f$;
this shows~\eqref{e:lower-bound-1}. The bound~\eqref{e:lower-bound-2} follows
since $R_h(E_0\pm i0)^*=R_h(E_0\mp i0)$.

\begin{figure}
\includegraphics[scale=1]{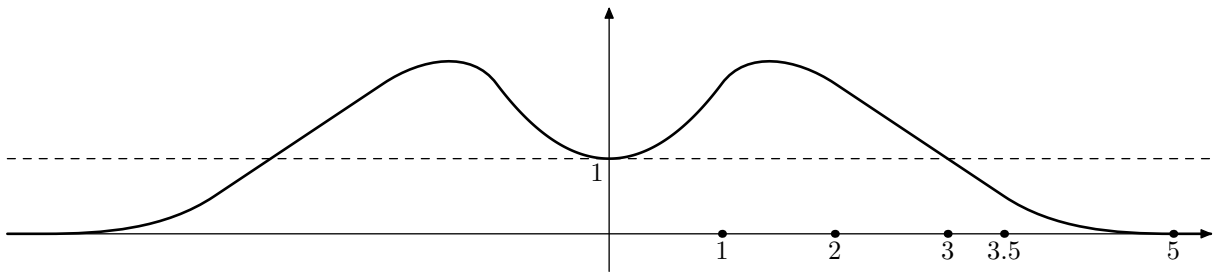}
\caption{The potential $V$ used in Theorem \ref{t:1d}.}
\label{f:V}
\end{figure}

The key component of the proof of Theorem~\ref{t:1d}
is the existence of \emph{quasimodes} for the operator $P-E_0$, namely functions that
satisfy $(P-E_0)v=\mathcal O(e^{-c/h})$:
\begin{lemm}
  \label{l:main}
There exist $h$-dependent families $E_0=E_0(h)\in\mathbb R$, and
$v=v(r;h)\in C^\infty([-R,R])$, such that $E_0=1+\mathcal O(h)$
and for some $C,c>0$,
$$
(P-E_0)v=0,\quad
\|v\|_{L^2(-3.5,3.5)}\geq C^{-1},\quad
\|v\|_{H^1_h(\{R-1<|x|<R\})}\leq Ce^{-c/h}.
$$
\end{lemm}
Here $H^1_h$ denotes the semiclassical Sobolev space where
in the standard definition $ D_x $ is replaced by $ h D_x $ -- 
see for instance \cite[\S 7.1, \S 8.3]{e-z}.

To derive Theorem~\ref{t:1d} from Lemma~\ref{l:main}, we take $\chi_0\in \CIc(-R,R)$
such that $\chi_0=1$ on $[-(R-1),R-1]$ and put
$$
u:=\alpha\chi_0 v,\quad
f=(P-E_0)u=\alpha[P,\chi_0]v,
$$
here the constant $\alpha=\alpha(h)$ is chosen so that
$\|\chi u\|_{L^2}=1$ and
we have $|\alpha|\leq C$.
We furthermore see that $\supp f\subset \{R-1<|x|<R\}$
and $\|f\|_{L^2}\leq e^{-c/h}$
(the constant $C$
can be absorbed into the exponential by replacing $c$ by a smaller constant
and taking $h$ small enough).

The rest of this section contains the proof of Lemma~\ref{l:main}. We take $\widetilde R(h)\geq R$,
$\widetilde R(h)=R+\mathcal O(h)$, to be chosen at the end of this
section in~\eqref{e:noquant}, and let $v$ be an eigenfunction
of $P$ on $[-\widetilde R(h),\widetilde R(h)]$ with Dirichlet boundary conditions with
eigenvalue $E_0$ close to the ground state $1+h$ of the quantum harmonic oscillator
$-h^2\partial_x^2+x^2+1$. The existence of such eigenvalue is given by the following
\begin{lemm}
  \label{e:qm-1}
For $h$ small enough and given $\widetilde R(h)\in [R,R+1]$, there exists $E_0\in\mathbb R$ and $v\in C^\infty([-\widetilde R(h),
\widetilde R(h)])$
such that
$$
\begin{gathered}
(P-E_0)v=0,\quad
v(\widetilde R(h))=v(-\widetilde R(h))=0;\\
\|v\|_{L^2(-\widetilde R(h),\widetilde R(h))}=1,\quad
\|v\|_{H^1_h(-\widetilde R(h),\widetilde R(h))}\leq C,\quad
E_0=1+h+\mathcal O(e^{-{1\over 10h}}).
\end{gathered}
$$
\end{lemm}
\begin{proof}
Define
$$
v_1(x):=h^{-1/4}e^{-{x^2\over 2h}},\quad
x\in [-1,1].
$$
Note that, since $P=-h^2\partial_x^2+x^2+1$ on $[-1,1]$, we have
$$
\begin{gathered}
(P-(1+h))v_1=0,\quad x\in [-1,1];\\
\|v_1\|_{L^2(-1/2,1/2)}\geq C^{-1},\quad
\|v_1\|_{H^1_h(\{1/2<|x|<1\})}\leq Ce^{-{1\over 10h}}.
\end{gathered}
$$
Now, take $\widetilde\chi\in \CIc(-1,1)$ such that $\widetilde\chi=1$ on $[-1/2,1/2]$.
Then
$$
\|\widetilde\chi v_1\|_{L^2}\geq C^{-1},\quad
\|(P-(1+h))\widetilde\chi v_1\|_{L^2}\leq Ce^{-{1\over 10h}},
$$
and
$\widetilde \chi v_1$ satisfies the Dirichlet boundary conditions at $\pm\widetilde R(h)$
(since it vanishes there).
Now, $P-(1+h)$ is self-adjoint on
$L^2([-\widetilde R(h),\widetilde R(h)])$ when Dirichlet boundary conditions are imposed.
Since the norm of its inverse is at least $C^{-1}e^{1\over 10h}$, we see that this
operator has an eigenvalue which is $\mathcal O(e^{-{1\over 10h}})$;
we denote the corresponding eigenvalue of $P$ by $E_0$ and the corresponding $L^2$
normalized eigenfunction by $v$.
Finally, to establish a bound on the $H^1_h$ norm of $v$ it suffices to multiply the
equation $(P-E_0)v=0$ by $\overline v$ and integrate by parts.
\end{proof}
Lemma~\ref{l:main} follows once we establish the following exponential bound on $v$:
\begin{equation}
  \label{e:key}
\|v\|_{H^1_h(\{3.5\leq |x|\leq \widetilde R(h)\})}\leq Ce^{-c/h}.
\end{equation}
We will show~\eqref{e:key} for positive $x$; the case of negative $x$ is handled similarly
(since $V$ is even, $v$ can be taken to be even as well).
The main idea is the following: if $v$ is not exponentially large in $1/h$
near, say, $x=2$ relative to its size on $[3.5,\widetilde R(h)]$, then
one expects $v$ to give an approximate Dirichlet eigenfunction
to the operator $P$ on $[2,\widetilde R(h)]$ with eigenvalue $E_0$.
However, then $E_0$ has to satisfy a quantization condition determined by
the behavior of $V$ on $[2,\widetilde R(h)]$; since $E_0=1+h+o(h)$,
one can choose $\widetilde R(h)$ to ensure that the quantization condition is not satisfied
and thus obtain a contradiction.

For $f_1,f_2\in C^\infty(\mathbb R)$ we define the (semiclassical)
Wronskian by
$$
W(f_1,f_2)=f_1\cdot h\partial_x f_2-f_2\cdot h\partial_x f_1,
$$
and note that
$$
h\partial_x W(f_1,f_2)=f_2\cdot (P-E_0)f_1-f_1\cdot (P-E_0)f_2.
$$
The interval $[2,R]$ can be split into three regions where the behavior of $v$ is different,
based on the sign of $V(x)-1$: the ``elliptic'' or classically forbidden region $[2,3)$,
where $v$ will grow exponentially in $h$ as $x$ decreases, the neighborhood of the turning
point $x=3$, and the ``hyperbolic'' region $(3,R)$, where the equation $(P(h)-E_0)v=0$ has
two solutions which are bounded as $h\to 0$.

We start with the hyperbolic region, considering the phase function
$$
\Phi(x):=\int_{4-E_0}^x \sqrt{E_0-V(y)}\,dy.
$$
Note that $\Phi$ is well-defined on $x\in [4-E_0,R+1]$, since
$\sqrt{E_0-V(y)}=\sqrt{y-(4-E_0)}$ for $y\in [4-E_0,3.5]$; in fact,
we have
\begin{equation}
  \label{e:phase-rel}
\Phi(x)={2\over 3}(x-(4-E_0))^{3/2}\quad\text{for }x\in [4-E_0,3.5].
\end{equation}
Define now the following \emph{WKB solutions}:
$$
v_\pm(x):=(E_0-V(x))^{-1/4}e^{\pm {i\Phi(x)\over h}},\quad
x\in [3.5, R+1],
$$
then we have uniformly in $x\in [3.5, R+1]$,
\begin{equation}
  \label{e:wkb-solve}
(P-E_0)v_\pm(x)=\mathcal O(h^2),\quad
W(v_+,v_-)(x)=-2i+\mathcal O(h).
\end{equation}
Denote
$$
\mathbf v(x)=(\mathbf v_1(x),\mathbf v_2(x)):=(W(v,v_+)(x),W(v,v_-)(x)),
$$
then
\begin{gather}
  \label{e:W-recovery}
v(x)={\mathbf v_2(x) \cdot v_+(x)-\mathbf v_1(x)\cdot v_-(x)\over W(v_+,v_-)(x)},\\
  \label{e:W-recovery-2}
h\partial_x v(x)={\mathbf v_2(x)\cdot h\partial_x v_+(x)-\mathbf v_1(x)\cdot h\partial_x v_-(x)\over W(v_+,v_-)(x)}.
\end{gather}
Since $(P-E_0)v=0$, we have
$$
h\partial_x W(v,v_\pm)=-v\cdot (P-E_0)v_\pm\,.
$$
From~\eqref{e:wkb-solve} and~\eqref{e:W-recovery} we see that for $x\in [3.5,R+1]$,
$$
|\partial_x \mathbf v(x)|\leq Ch|\mathbf v(x)|.
$$
Therefore,
\begin{equation}
  \label{e:wkb-result}
\mathbf v(x)=\mathbf v(3.5)(1+\mathcal O(h)),\quad
x\in [3.5,\widetilde R(h)].
\end{equation}
This and~\eqref{e:W-recovery}, \eqref{e:W-recovery-2} show that
\begin{equation}
  \label{e:W-bound}
\|v\|_{H^1_h(3.5,\widetilde R(h))}\leq C|\mathbf v(3.5)|.
\end{equation}
The final component of the proof is the following solution in the region $[2,3.5]$
which describes the transformation from the hyperbolic to the elliptic region via the
turning point, and is exponentially decaying in
the elliptic region. Since $V(x)=4-x$ in this region, the solution is given by an Airy function,
and its properties are as follows:
\begin{lemm}
  \label{l:airy}
There exists a solution $w(x)$ to the equation $(P-E_0)w=0$ for $x\in [2,3.5]$
such that $\|w\|_{H^1_h(2,2.5)}\leq Ce^{-c/h}$ for some constants $C,c>0$ and 
\begin{equation}
  \label{e:airy}
\begin{pmatrix}
w(x)\\
h\partial_x w(x)
\end{pmatrix}
= e^{i\pi \over 4}
\begin{pmatrix}
v_+(x)\\
h\partial_x v_+(x)
\end{pmatrix}
- e^{-{i\pi\over 4}}
\begin{pmatrix}
v_-(x)\\
h\partial_x v_-(x)
\end{pmatrix}+O(h),
\quad
x\in [3.25, 3.5].
\end{equation}
\end{lemm}
\begin{proof} The solution $w$ is given by
$$
w(x)=2i \sqrt\pi h^{-1/6} \Ai(h^{-2/3}(4-E_0-x)),
$$
and its properties follow from the following asymptotic formul\ae{} for the Airy function $\Ai$
as $y\to +\infty$:
$$
\begin{aligned}
\Ai(y)={y^{-1/4}\over 2\sqrt \pi} \exp\Big(-{2\over 3}y^{3/2}\Big) (1+\mathcal O(y^{-3/2})),\\
\Ai(-y)={y^{-1/4}\over \sqrt\pi} \bigg(\sin\Big({2\over 3}y^{3/2}+{\pi\over 4}\Big)+\mathcal O(y^{-3/2})\bigg),
\end{aligned}
$$
and similar formul\ae{} for its derivatives,
see for example~\cite[(7.6.20), (7.6.21)]{ho1}.
\end{proof}
We are now ready to finish the proof of~\eqref{e:key}. Since
$(P-E_0)v=(P-E_0)w=0$ on $[2,3.5]$, the Wronskian
$W(v,w)$ is constant on this interval.
Using the estimates $\|v\|_{H^1_h(2,2.5)}\leq C$,
$\|w\|_{H^1_h(2,2.5)}\leq Ce^{-c/h}$, we see that
$$
|W(v,w)|\leq Ce^{-c/h}.
$$
Now, computing the same Wronskian at $x=3.5$ and using~\eqref{e:airy}, we get
$$
W(v,w)=e^{i\pi\over 4}\mathbf v_1(3.5)-e^{-i\pi\over 4}\mathbf v_2(3.5)+\mathcal O(h)|\mathbf v(3.5)|,
$$
It remains to prove that, for a certain choice of $\widetilde R(h)$ independent of $E_0$, we have
\begin{equation}
  \label{e:funal}
|e^{i\pi\over 4}\mathbf v_1(3.5)-e^{-i\pi\over 4}\mathbf v_2(3.5)|\geq C^{-1}|\mathbf v(3.5)|.
\end{equation}
Indeed, in this case $|\mathbf v(3.5)|\leq Ce^{-c/h}$, which together with~\eqref{e:W-bound}
gives~\eqref{e:key}.

Using~\eqref{e:wkb-result}, we rewrite~\eqref{e:funal} as follows:
\begin{equation}
  \label{e:funal2}
|e^{i\pi\over 4}\mathbf v_1(\widetilde R(h))-e^{-i\pi\over 4}\mathbf v_2(\widetilde R(h))|\geq C^{-1}|\mathbf v(\widetilde R(h))|.
\end{equation}
Since $v$ satisfies the Dirichlet boundary condition at $\widetilde R(h)$, we have
$$
W(v,v_\pm)(\widetilde R(h))=-E_0^{-1/4} e^{\pm {i\Phi(\widetilde R(h))\over h}}\cdot h\partial_x v(\widetilde R(h)),
$$
so that $|\mathbf v_1(\widetilde R(h))|,|\mathbf v_2(\widetilde R(h))|\geq C^{-1}|\mathbf v(\widetilde R(h))|$ and
$$
{e^{-{i\pi\over 4}}\mathbf v_2(\widetilde R(h))\over e^{i\pi\over 4}\mathbf v_1(\widetilde R(h))}=\exp\Big(-{i\over h}\big(2\Phi(\widetilde R(h))+\pi h/2\big)\Big).
$$
To prove~\eqref{e:funal2}, we choose $\widetilde R(h)=R+\mathcal O(h)$, $\widetilde R(h)\geq R$, so that
\begin{equation}
  \label{e:noquant}
\min_{j\in\mathbb Z}\big|\Phi(\widetilde R(h))+(j+1/4)\pi h\big|\geq {\pi h\over 4};
\end{equation}
this can be done independently of $E_0$, since $E_0=1+h+\mathcal O(e^{-{1\over 10h}})$ and thus
$$
\Phi(\widetilde R(h))=\int_{3-h}^5 \sqrt{1+h-V(y)}\,dy+\sqrt{1+h}(\widetilde R(h)-5)+\mathcal O(e^{-{1\over 10h}}).
$$

\section{A general argument}
\label{s:general}

Suppose that $ P ( h) $ is an operator satisfying the general
assumptions of \cite{nsz}, that is a black box self-adjoint 
operator, 
close to the Laplacian and having analytic coefficients near
infinity and with a barrier at energy $ V_0 $. 
 (A barrier separates the
interaction region from infinity -- see~\eqref{eq:barriera} and Fig.~\ref{f1}).
We assume that $ \RR^n \setminus B ( 0, R_0 ) $ is contained
in the ``outside'' of the black box and the trapped set at energy $ V_0 $.
We also assume the Hilbert space on which the operator acts, 
$  \mathcal H  = \mathcal H_{R_0} \oplus L^2 ( \RR^n \setminus 
B ( 0 , R_0 ) $, 
is equipped with an involution $ u \mapsto \bar u $ equal to 
complex conjugation on $ L^2 ( \RR^n \setminus B ( 0 , R_0 ) ) $ and
satisfying  $ \overline { z u } = \bar z \bar u $, $ z \in \CC $. The
abstract reality assumption on $ P ( h ) $ reads
$ \overline {P ( h ) u} = P ( h ) \overline u $. 

An example to keep in mind is given by the operator
\[ P ( h ) = - h^2 \Delta_g + V : H^2 ( \RR^n ) \to L^2 ( \RR^n ) , 
\]
where the measure on $ L^2 $ is obtained from the Riemannian 
metric and $ \mathcal H := L^2 ( \RR^n ) $. 
The potential $ V ( x ) $ and  the metric coefficients $ g^{ij}
( x )  $ are smooth, extend analytically to $ \Omega := \{ z \in \CC^n
: | z | \geq R_1 , \ \ | \Im z | \leq \delta | z| \}$, and 
\begin{equation}
\label{eq:gdv}    g^{ij} ( z ) - \delta^{ij} \rightarrow 0 , \ \  V ( z )
\rightarrow 0 , \ \ | z | \to \infty , z \in \Omega . \end{equation}
The trapped set,  $ K_E $,  at energy $ E > 0  $ is then defined by 
\begin{equation}
\label{eq:KE}  K_E :=  \{ ( x , \xi ) \in \RR^n \times \RR^n : p ( x, \xi ) = E, \ \ e^{ tH_p } ( x, \xi ) \not
\to \infty  , \ \ t \to \pm \infty \}
, \end{equation} 
where 
$ H_p := \sum_{ j=1}^n \partial_{\xi_j} p\cdot \partial_{x_j }
- \partial_{x_j} p\cdot \partial_{ \xi_j }$ and 
$  p( x, \xi) := \sum_{
  i,j=1}^n g^{ij} ( x ) \xi_j \xi_i + V ( x ) $.

The assumption on the interaction region in this case means that $ \pi ( K_{V_0}  ) \subset B (
0 , R_0 ) $. The barrier assumption means that
\begin{equation}
\label{eq:barriera}
p^{-1} ( V_0 ) = K_{V_0} \cup \Sigma_{V_0 } , \ \ \Sigma_{V_0 }  \cap
K_{V_0} = \emptyset , \ \ \Sigma_{V_0 } \text{ is closed.}
\end{equation}

\begin{figure}
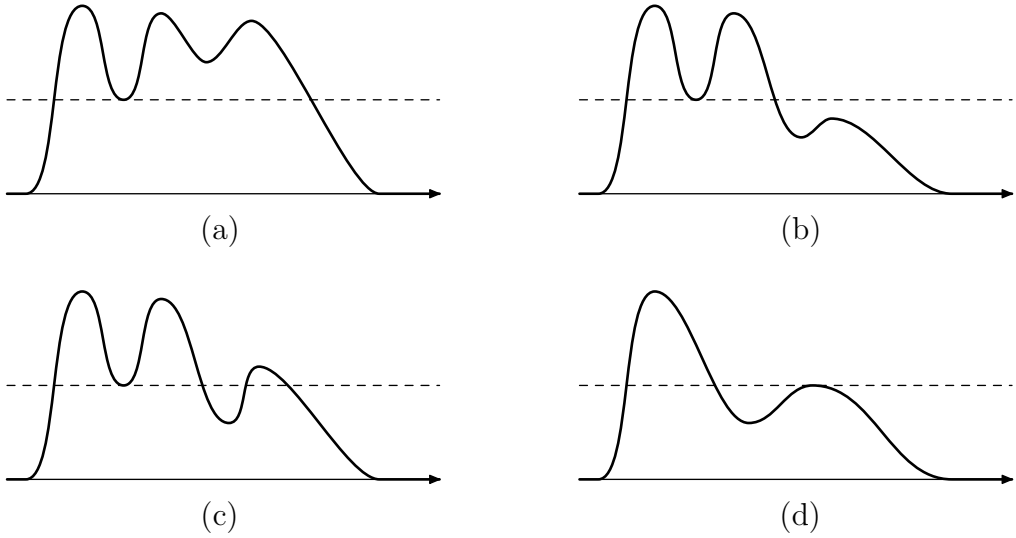

\includegraphics{resolve.1}
\qquad\qquad
\includegraphics{resolve.2}
\hbox to\hsize{\hss (a)\hss\hss(b)\hss}
\vskip.2in
\includegraphics{resolve.3}
\qquad\qquad
\includegraphics{resolve.4}
\hbox to\hsize{\hss (c)\hss\hss(d)\hss}
\caption{Examples of one-dimensional potentials satisfying: (a) condition~\eqref{e:as-intro}
and hence~\eqref{eq:barriera} and~\eqref{eq:KV0};
(b) conditions~\eqref{eq:barriera} and~\eqref{eq:KV0}, but not~\eqref{e:as-intro};
(c) condition~\eqref{eq:barriera}, but not~\eqref{eq:KV0};
(d) neither~\eqref{eq:barriera} nor~\eqref{eq:KV0}.
The dashed line corresponds to $V_0$. In particular, examples (a) and (b) satisfy the assumptions of
Theorems \ref{t:1} and \ref{t:2}.}
\label{f1}
\end{figure}

The more general black box setting allows obstacle problems and other
geometric situations. However, the barrier assumption cannot be satisfied
for connected manifolds without having a nontrivial
potential $ V $. 

We denote by $\Res(P(h))$ the set of resonances of $P(h)$ (in an $h$-independent neighborhood of
$\{\Re z>0,\Im z>0\}$).

\begin{thm}
\label{t:1} 
Let $ P ( h ) $ satisfy the general assumptions above. 
Suppose that $ z_0 = z_0 (h) \in \Res ( P ( h)) $,  $ \Re z_0  = V_0  +
\mathcal O ( h ) $,  $ z_0 $
is simple, and 
\begin{equation}
\label{eq:assz0} 
 | \Im z_0 | = {\mathcal O} ( h^\infty ) , \ \ \ 
  d ( z_0 ( h ) , \Res ( P ( h ) ) \setminus \{z_0 ( h ) \} ) > h^N
, \end{equation}
for 
some $ N $.  
Suppose that $ \chi $ and $ \psi $ are given by \eqref{eq:cutoffs}
with $ R > R_0 $. 
Then there exist $ C_0 > 0 $ and $ h_0
$ such that for $ 0 < h < h_0 $, 
\begin{equation}
\label{eq:res1}
\|  \chi ( P ( h ) - \Re z_0 - i 0 )^{-1} \chi \|_{{\mathcal H} \to {\mathcal H} }
\geq \frac{ 1 }{ C_0 | \Im z_0 | } , 
\end{equation}
and
\begin{equation}
\label{eq:res2}
\|  \psi ( P ( h ) - \Re z_0 - i 0 )^{-1} \chi \|_{{\mathcal H} \to {\mathcal H} }
\geq \frac{ 1 }{ C_0 \sqrt{ | \Im z_0 | h } }.
\end{equation}
\end{thm}

\noindent\textbf{Remark.}
As stated in the introduction, it is \eqref{eq:res2} that seems to be the novel aspect. The presence
of the square root is (morally) consistent with the results of
\cite[Lemma A.2]{Bz} and of \cite[Theorem 2]{da-va}.

\begin{proof}
We only present the proof of \eqref{eq:res2}.
Using the involution $ u \mapsto \bar u $ we define $ (  u
\otimes v ) f  := u \langle f , \bar v \rangle $, where we use the
inner product on the black box Hilbert space. 
Since $ z_0 $ is simple, we have
\[  \psi ( P ( h ) - z )^{-1} \chi = \frac{ \psi u \otimes \chi u } {
  z - z_0 } + \psi R_{z_0} ( z , h ) \chi , \]
where $u$ is the corresponding normalized resonant state and $ R_{z_0} ( z , h ) $  is holomorphic in
\[  [ \Re z_0 - h^N   , \Re z_0 + h^N ] + i
( - h^N , \infty ) .\]

From \cite[(5.1)]{nsz} and \cite[(1.12)]{bm} (or \cite[(8.18)]{nz3})
we see that 
\begin{equation}
\label{eq:chiu}  \| \chi u \|_{{\mathcal H}} = 1 + {\mathcal O}( h^\infty ), 
\ \ u = \chi u + {\mathcal  O}(h^\infty)_{L^2_{\rm{loc}}}
\,. \end{equation}

Using the maximum principle as in~\cite[Lemma 2]{tz} and the estimates on the resolvent in~\cite[Lemma
1]{tz} we see that 
\begin{equation}
\label{eq:reso}  \| \chi R_{z_0} ( \Re z_0  , h ) \psi \|_{{\mathcal H} \to {\mathcal H} } = \mathcal
O ( h^{-M} ), \end{equation}
for some $ M$. Hence to obtain \eqref{eq:res2} we need to estimate
\begin{equation}
\label{eq:tens}    \frac{1} { |\Im  z_0 | }\| \psi u \otimes \chi u \|_{ \mathcal H \to  {\mathcal H} } =  \frac{ \| \psi u \|_{\mathcal H} \| \chi u \|_{{\mathcal H}} } { |\Im  z_0 |
} , \end{equation}
from below.

To estimate  $ \| \psi u \|_{\mathcal H} $ from below we write
\[\begin{split}
0 &= \Im  \langle  (P ( h ) -z_0)u, 
\indic_{ B ( 0 , R +t   ) } u \rangle_{\mathcal H}  
= \Im \langle P ( h ) u , \indic_{ B ( 0 , R +t   ) } u 
\rangle_{\mathcal H} - \Im z_0 \| \indic_{ B ( 0 , R + t ) } u \|^2_{\mathcal H} .
\end{split}\]
Since $ P ( h ) $ is self-adjoint on $ \mathcal H $ and acts as
a symmetric second order operator on $ \CIc ( \RR^n \setminus 
B ( 0 , R_0 )) $, we obtain
\[ \Im z_0 \| \indic_{ B ( 0 , R + t ) } u \|^2_{\mathcal H}
= - \Im \int_{ \partial B ( 0 , R + t ) } \bar u 
N ( x, h D ) u d S( x ) , \]
where $N(x,hD_x)$ is a first order semiclassical differential operator. For example, if $  P ( h ) = -h^2\Delta +V$, then 
$  N(x,hD_x) = { i} (x/|x|) \cdot hD_x  $.

Since from \eqref{eq:chiu}, 
$ \|\indic_{ B ( 0 , R +t ) }  u  \|^2_{\mathcal H} = 1 + {\mathcal O}_t ( h^\infty
)$, 
we see that
\[ \begin{split}  \textstyle{\frac13} |\Im z_0 | & \leq h \int_{\frac13}^{\frac23} \int_{ \partial B ( 0 , R+ t ) }
 | u| |N(x,hD_x) u| d S ( x ) dt
 \\ & 
=  h \int_{ B( 0 , R + \frac23 )   \setminus B ( 0 , R  + \frac13 ) } |
u | | N(x,h D_x) u  | dx \\
& \leq C' h \int_{ B ( 0 , R + 1 ) \setminus B ( 0 , R ) }  | u |^2 dx \leq C' h \| \psi u \|_{ \mathcal H}^2 , \end{split} \]
where we used the equation $(P-z_0)u=0$  and elliptic estimates to control  the first order term term $ N(x,hD_x) u $.
This shows that
\[ \| \psi u \|_{\mathcal H} \geq \sqrt{ { |\Im z_0| } / { C h} } ,\]
which combined with \eqref{eq:reso},\eqref{eq:tens} and
\eqref{eq:chiu} completes the proof of \eqref{eq:res2}.
\end{proof}


\section{A metric example}
\label{s:ame}

We start by using Theorem \ref{t:1} to obtain a generalization 
of Theorem \ref{t:1d} to higher dimensions and to more
general potentials: 

\begin{thm}
\label{t:2}
Consider a Schr\"odinger operator $ P ( h) = -h^2 \Delta + V $ on $L^2(\mathbb R^n)$ where $ V $
satisfies \eqref{eq:gdv} and \eqref{eq:barriera}. Suppose also that
\begin{equation}
\label{eq:KV0}  K_{V_0 } = \{ ( x_0 , 0 ) \},  \ \ V' ( x_0 ) = 0 ,  \ \ V'' (
x_0) > 0 . \end{equation}
 Then there exists $ z_0 $ satisfying the assumptions of Theorem
\ref{t:1} with 
\begin{equation}
\label{eq:t2}  d ( z_0 , \Res ( P ( h ) ) \setminus \{ z_0 \} ) > h
/ C, \ \ \ \ | \Im z_0 |< e^{ -c_0 /h } \,. \end{equation}
In particular in the notation of \eqref{eq:res2}, 
\begin{equation}
\label{eq:inpart}   \|  \psi ( P ( h ) - \Re z_0 \pm i 0 )^{-1} \chi \|_{L^2 \to L^2 } \geq \exp\frac{c}h ,  \end{equation}
for $ 0 < h < h_0 $ and  some $ c > 0 $.
\end{thm}
\begin{proof}
The existence of $ z_0 $ follows from Corollary in \cite[\S 5]{nsz}.
The reference operator $ P^\sharp ( h ) $ there 
can be chosen as $ P^\sharp ( h ) = - h^2 \Delta + V^\sharp ( x ) $,
where $ V^\sharp ( x ) = V ( x ) $ in a small neighbourhood of 
$ x_0 $ where $ x_0 $ is the only critical point 
and $ V^\sharp ( x) > V ( x_0) + \epsilon $, $ \epsilon > 0$,
outside of that neighbourhood.
Since the eigenvalue of $ P^\sharp ( h ) $ corresponding to the minimum $V_0= V (x_0) $ is separated
from other eigenvalues by $ h/C $ (see for instance 
\cite{hSj1} and references given there)
 the same corollary shows the 
separation from other resonances.
\end{proof}

\noindent
{\bf Remarks.} 
1. The condition \eqref{e:as-intro} implies that \eqref{eq:barriera}
and \eqref{eq:KV0} hold (since $H_p(x\cdot \xi)>0$ on $\{p=V_0\}$
except at $x=\xi=0$), but the converse is not true~-- see Figure~\ref{f1}.

\noindent
2. When $ V $ is analytic and satisfies
certain ``well-in-the-island'' hypotheses, Theorem \ref{t:2} follows from the work 
of Helffer--Sj\"ostrand \cite{hsj} and under these stronger assumptions   Theorem \ref{t:1} can then
be proved in the same way using the earlier results of
G\'erard--Martinez \cite{gm} in place of the results of
\cite{nsz}.

\noindent
3. For $ P (h ) = -h^2 \Delta + V $, 
and for $ E$'s satisfying \eqref{eq:KV0} (with $ V_0 = E$),
a result of 
Nakamura \cite[Proposition 4.1]{na} and \cite[Corollary, \S 5]{nsz}
show that
\[  \| \chi ( P ( h ) - E - i 0)^{-1} \chi \|_{L^2 \to L^2} \leq
C h^{-q} , \ \  | E - z_j ( h ) | \geq h^q , \]
where $ q \geq 1 $ and $ z_j ( h ) $ are the resonances of 
$ P ( h ) $. Since the density of $ \Re z_j ( h ) $ satisfies
a Weyl law, this means that the bound is $ \mathcal O ( h^{-q} ) $,
outside of a set of measure $ \mathcal O ( h^{q-n} ) $, $ q > n $.

The example in Theorem \ref{t:2} can be used directly to 
obtain examples of resolvent growth 
for asymptotically conic metrics of the 
type studied by Rodnianski--Tao \cite{rt}. 

\begin{thm}
\label{t:3}
Let $(M,g)$ be the following Riemannian manifold:
$$
M=\mathbb R_x \times\mathbb S^{n-1}_\theta,\quad
g=dx^2+V(x)^{-1}\,d\theta^2,  \ \ n > 1 ,
$$
where $d\theta^2$ is the round metric on the sphere of radius 1 and
$V(x)\in C^\infty(\mathbb R;(0,\infty))$ is a function satisfying
the assumptions of Theorem \ref{t:2} and 
$$ V ( x ) = \frac{1}{x^2} , \ \ |x| \geq R_0 .$$
Put
\[  \chi ( x ) = \indic_{ |x | < R_0  } , \ \ \ \psi ( x ) =
\indic_{  R-1 < |x| < R+1 } , \ \ R > R_0 . \]
Then 
there exists a sequence $ \lambda_k \to \infty $ such that
\begin{equation}
\label{eq:rota}
\| \psi ( -\Delta_g - \lambda_k \pm i0)^{-1} \chi \|_{L^2 ( 
M ) \to L^2 (  M ) } \geq 
\exp ( c \sqrt {\lambda_k} ), 
\end{equation}
for some constant $ c>0 $.
\end{thm}
\begin{proof}
In the $(x,\theta)$ coordinates, the Laplacian $\Delta_g$ has the form
$$
\Delta_g=\partial_x^2-{(n-1)V'(x)\over 2V(x)}\partial_x+V(x)\Delta_S.
$$
Here $\Delta_S$ is the Laplacian on $\mathbb S^{n-1}$.
For $k\geq 0$, let $Y_k(\theta)$ be (any) spherical harmonic of order $k$, i.e. a smooth function on $\mathbb S^{n-1}$ such that
$$
(-\Delta_S-k(k+n-2))Y_k=0,\quad
\|Y_k\|_{L^2(\mathbb S^{n-1})}=1,
$$
see for example~\cite[\S 17.2]{ho1} for the spectrum of $\Delta_S$. Then for $u(x)\in C^\infty(\mathbb R)$ and $\lambda\in\mathbb R$, we have
$$
-\Delta_g (u(x)Y_k(\theta))=\Big(-\partial_x^2+{(n-1)V'(x)\over 2V(x)}\partial_x+k(k+n-2)V(x)\Big)u(x) Y_k(\theta).
$$
Put
$ h_k:=\big(k(k+n-2)\big)^{-1/2}$ 
so that
$$
h_k^2(-\Delta_g-\lambda)(u(x)Y_k(\theta))=(P(h_k) - h^2_k \lambda  )u(x)Y_k(\theta),
$$
where
$$
P(h):=-h^2\partial_x^2+{(n-1)V'(x)\over 2V(x)}h^2\partial_x+V(x) \,. 
$$
Let $ R ( \lambda ) := ( - \Delta_g - \lambda )^{-1} $ for $
\lambda\not\in [0,+\infty) $. It follows that
\begin{gather}
\label{eq:Rla}
\begin{gathered}    R ( \lambda ) = \sum_{k \in\NN} h_k^2 ( P ( h_k ) - h_k^2 \lambda )^{-1}
\otimes \Pi_k : L^2 ( M ) \to L^2 ( M ) , 
  \\  L^2 ( M ) \simeq L^2 ( \RR , V ( x )^{
 -\frac{n-1}2 } dx )\otimes L^2 (
\SP^{n-1} ) , \end{gathered}
\end{gather}
where $ \Pi_k : L^2 ( \SP^{n-1} )  \to L^2 ( \SP^{n-1} ) ) $ is the
orthogonal projection onto the space of spherical harmonics of order $
k $. The operator $ R ( \lambda ) : \CIc ( M ) \to \CI ( M )  $ continues meromorphically to $
\Im \lambda \leq 0 $, and $ ( P ( h_k ) - h_k^2 \lambda )^{-1} : \CIc (
\RR ) \to \CI ( \RR )$ continues meromorphically for each $ k $. Hence
\eqref{eq:Rla} is valid for $ \Im \lambda \leq 0 $, 
with the operator acting
on  $ \CIc ( M ) \simeq \CIc ( \RR ) \otimes \CIc ( \SP^{n-1} ) $.

Hence, 
\[ \begin{split}
 \| \psi ( - \Delta_g - \lambda\pm i0 )^{-1} \chi\|_{ L^2 ( M ) \to L^2 ( M ) }
&  = \| \psi R ( \lambda\pm i0 ) \chi \|_{ L^2 ( M ) \to L^2 ( M )} \\
& = \sup_{
  k \in \NN } h_k^2 \| \psi ( P ( h_k ) - h_k^2 \lambda  \pm i0)^{-1} \chi
\|_{ L^2_x \to L^2_x } , 
\end{split} \]
where $ L_x^2 := L^2 ( \RR , V( x )^{-\frac{n-1}2} dx ) $. 
We now apply Theorem \ref{t:2} to $ P ( h_k ) $ and put
$ \lambda_k = \Re z_0 ( h_k )/h_k^2 $. The estimate 
\eqref{eq:rota} follows from \eqref{eq:inpart}.
Theorem~\ref{t:2} applies to the operator $P(h)$ despite the presence of a first order term,
as this term is of order $\mathcal O(h)$ in the semiclassical calculus and thus does not affect
the classical Hamiltonian flow $H_p$, and the results of~\cite{nsz} and Theorem~\ref{t:1} apply
to a wide class of semiclassical differential operators including $P(h)$.
\end{proof}

\def\arXiv#1{\href{http://arxiv.org/abs/#1}{arXiv:#1}}

\end{document}